\documentclass{amsart}
\usepackage{amssymb}
\usepackage{amsthm}

\newcommand{\pmn}{\par\medskip\noindent}
\newcommand{\pbn}{\par\bigskip\noindent}
\newtheorem{theorem}{Theorem}[section]
\newtheorem{lemma}{Lemma}[section]
\theoremstyle{definition}

\DeclareMathAlphabet{\mathchorus}{OT1}{cmtt}{m}{n}  

\newcommand{\mc}[1]{\mathchorus{#1}} 

\newcommand{\bv}{\begin{pmatrix}}  
\newcommand{\ev}{\end{pmatrix}} 
\newcommand*{\defeq}{\stackrel{\text{def}}{=}} 
\newcommand*{\cat}{\mkern 2mu _{\smallsmile}} 

\DeclareMathOperator{\cp}{\Lambda}

\begin{document}
\title{ 
\hspace{0.055\textwidth}The Hadamard product and recursively defined sequences }
\author{Sergei Kazenas }
\begin{abstract}
In this paper the approach to obtaining nonrecurrent formulas for some recursively defined sequences is illustrated.
The most interesting result in the paper is the formula for the solution of quadratic map-like recurrence.
Also, some formulas for the solutions of linear difference equations with variable coefficients are obtained.
At the end of the paper, some integer sequences associated with a quadratic map are considered.
\end{abstract}

\email{kazenas@protonmail.com} \maketitle

\hspace{0.77\textwidth} {\it{\large{to Margarita}}}

\section{Introduction}

We begin with some notation. Let $ \mc{b}$ be a row vector and $\mc{a}$ be a finite row vector;
$(\mc{b})_j$ denotes $j$th element of vector $\mc{b}$,
$|\mc{b}|$ denotes length of $ \mc{b}$;
$\mc{a} \cat \mc{b} \defeq \bv (\mc{a})_1 ,\ldots , (\mc{a})_{|\mc{a}|} \ , (\mc{b})_1 , \ldots \ev $,
$(\mc{b})_{l}^{m} \defeq \bv (\mc{b})_j \ev _{j=l}^{m}$,
$|\mc{b}|_x \defeq \sum_{j=1}^{|\mc{b}|} (\mc{b})_j x^{j-1}$, 
$\mc{1}_{m} \defeq \bv 1 \ev_{j=1}^{m}$,
$\mc{0}_{m} \defeq \bv 0 \ev_{j=1}^{m}$,
$[ l, m ]_j \defeq ({\mc{1}}_\infty \times ( {\mc{0}}_{l} \cat {\mc{1}}_{m} ) )_j$,
where $\times$ denotes the Kronecker product,
$ l \in \mathbb{N}$,  
$m \in \mathbb{N} \cup \{\infty\}$.
Note that the last function can be expressed by the ceiling function: 
$[ l, m ]_j = \lceil (j + m)/(l + m)\rceil - \lceil j/(l + m)\rceil$.

\pmn Also, we use ordinary notation to denote the corresponding entrywise operations. 
For example, $\mc{a}^2$ expresses the Hadamard square: $\mc{a}^2 = \bv {(\mc{a})_j}^2 \ev _{j=1}^{|\mc{a}|}$. 

\pmn It should be noted that there are many papers on sequences generated by linear difference equations with variable coefficients. 
See, for examle, \cite{Pop1987}, \cite{Kit1993}, \cite{Mal2000}.
The simple approach illustrated here involves constructing for each such sequence a corresponding recursive vector sequence, which can 
be explicitly expressed using the following property of Hadamard product: 
$(\mc{a}\mc{c})\ast (\mc{b}\mc{d})=(\mc{a} \ast \mc{b})(\mc{c} \ast \mc{d})$,
where $|\mc{c}| = |\mc{a}|$, $|\mc{d}| = |\mc{b}|$ and $\ast \in \{ \cat , \times  \}$.

\section{Linear recurrences}

First we use this property with respect to concatenation.

\begin{theorem}
If $x_1$, $x_2$ are arbitrary numbers, $a_n$, $b_n$ are arbitrary number sequences 
and $x_n = a_n x_{n-1} + b_n x_{n-2}$ for $n\geqslant 3$, then
$$
x_n = \sum_{j=1}^{f_n} ((x_1 - x_2)(\mc{f})_j + x_2) \biggl( \prod_{\substack{3 \leqslant k \leqslant n \\ (\mc{f}_k)_j = 1}} b_k  \biggr) 
\biggl( \prod_{\substack{3 \leqslant k \leqslant n \\ (\mc{f}_k + \mc{f}_{k+1})_j = 0}} a_k  \biggr) \text{,}
$$
where $f_n$ is $n$th Fibonacci number, $\mc{f} = \bv 0,1,0,0,\ldots \ev $ is infinite Fibonacci word;
$\mc{f}_k$ is obtained from $\mc{f}$ by replacing each entry of zero with $f_{k-1}$ zeros and each entry of one with $f_{k-2}$ ones.
\end{theorem}
\begin{proof} Define vectors: $\mc{x}_1 = \bv x_1 \ev$, $\mc{x}_2 = \bv x_2 \ev$, 
$\mc{x}_n = (b_n \mc{x}_{n-2}) \cat (a_n \mc{x}_{n-1})$ for $n\geqslant 3$ and
$\mc{y}_1 = \mc{x}_1$, $\mc{y}_n = \mc{y}_{n-1} \cat \mc{x}_n$, 
$\mc{p}_n = \mc{1}_{f_{n+1}-1} \cat (b_n \mc{1}_{f_{n-2}}) \cat  (a_n \mc{1}_{f_{n-1}}) $ 
for $n\geqslant 2$.
\pmn Let $\cp_k \mc{b} \defeq \mc{b} \cat (\mc{b})_k^{|\mc{b}|}$. 
We have $\mc{y}_n = \mc{p}_n \cp_{f_{n-1}} \mc{y}_{n-1}$ for $n\geqslant 3$, 
from which it follows that $\mc{y}_n = \mc{y}_{2,n} \prod_{k=3}^n\mc{p}_{k,n}$, 
where $\mc{y}_{2,n} = \cp_{f_{n-1}}\cp_{f_{n-2}}\ldots\cp_{f_2}\mc{y}_2$,
$\mc{p}_{k,n} = \cp_{f_{n-1}}\cp_{f_{n-2}}\ldots\cp_{f_k}\mc{p}_k$.
\pmn Partition $\mc{y}_{2,n} = \mc{y}_1' \cat \ldots \cat \mc{y}_n'$ such that $|\mc{y}_i'| = f_i$,
then $\mc{y}_1' = \mc{x}_1$, $\mc{y}_2' = \mc{x}_2$, $\mc{y}_n' = \mc{y}_{n-2}' + \mc{y}_{n-1}'$ for $n\geqslant 3$.
Similarly partition $\mc{p}_{k,n} = \mc{p}_{k,1}' \cat \ldots \cat \mc{p}_{k,n}'$ such that $|\mc{p}_{k,i}'| = f_i$,
then $\mc{p}_{k,i}'=\mc{1}_{f_i}$ for $1 \leqslant i\leqslant k-1$, $\mc{p}_{k,k}'=(b_k \mc{1}_{f_{k-2}}) \cat (a_k \mc{1}_{f_{k-1}})$,
$\mc{p}_{k,i}' = \mc{p}_{k,i-2}'\cat \mc{p}_{k,i-1}'$ for $k+1 \leqslant i\leqslant n$.
\pmn Note that $(\mc{x}_n)_{-} = (\mc{y}_n')_{-} \prod_{k=3}^n(\mc{p}_{k,n}')_{-}$,
where by $(\mc{a})_{-}$ we denote the vector composed of elements of $\mc{a}$ in reverse order.
Now $(\mc{y}_n')_{-}$ и $(\mc{p}_{k,n}')_{-}$ can be expressed in terms of infinite generalized Fibonacci words:
$(\mc{y}_n')_{-} = (x_1 - x_2) ( \mc{f} )_1^{f_n}  + x_2 \mc{1}_{f_n}$,
$(\mc{p}_{k,n}')_{-} = ( \mc{f}_{k+1} + b_k \mc{f}_k + a_k (\mc{1}_{f_n}-\mc{f}_{k+1}-\mc{f}_k))_1^{f_n}$. 
\pmn Finally using $x_n = |(\mc{x}_n)_{-}|_1$ we get the result. \end{proof}
\pmn The same sequence can be expressed with help of the Kronecker product.

\begin{theorem}
If $x_1$, $x_2$ are arbitrary numbers, $a_n$, $b_n$ are arbitrary number sequences 
and $x_n = a_n x_{n-1} + b_n x_{n-2}$ for $n\geqslant 3$, then
$$
x_n = \sum_{\substack{2^{n-2} + 1 \leqslant j \leqslant 2^{n-1} \\  \vartheta(2^{n-1}-j+1)= 1} } 
((x_2-x_1)[ 1,1 ]_j + x_1)
\prod_{\substack{0 \leqslant k \leqslant n-3 \\ [ 3\cdot2^k,2^k ]_j =1 }}a_{k+3}
\prod_{\substack{0 \leqslant k \leqslant n-3 \\ [ 2^k,2^k ]_j =0 }}b_{k+3} \text{,}
$$
where $\vartheta (n) \normalfont{\defeq} \prod_{k=0}^{\infty} (1-[ 3\cdot2^k,2^k ]_n)$.
\end{theorem}
\begin{proof} Define vectors: $\mc{r}_1 = \bv x_1 , x_2 \ev$, $\mc{r}_n = (\mc{1}_2 \times \mc{r}_{n-1})(\mc{h}_n \times \mc{1}_{2^{n-2}})$ 
for $n\geqslant 2$, where $\mc{h}_n = \bv 0,1,b_{n+1},a_{n+1}\ev$.
It can be easily shown that $|(\mc{r}_n)_{2^{n-1}+1}^{2^n}|_1 = x_{n+1}$.
Solving  the recurrence equation we get: 
$r_n = (\mc{1}_{2^{n-1}} \times \mc{r}_1) \prod_{k=2}^{n} (\mc{1}_{2^{n-k}} \times \mc{h}_k \times \mc{1}_{2^{k-2}})$.
Taking into account that $\mc{h}_k = \bv 0,1,1,1 \ev \bv b_{k+1},1,b_{k+1},1\ev \bv 1,1,1, a_{k+1}\ev $ 
and doing some calculations we get the result.\end{proof}

\pmn The following lemma allows us to generalize the result to the nonhomogeneous case.

\begin{lemma}
If $\mc{x}_1$ is arbitrary vector,
$\mc{b}_n$ is $0,1$-vector sequence,
$\mc{a}_n$ and $\mc{c}_n$ are such that
$|\mc{a}_n| = |\mc{c}_n| = |\mc{x}_1| \prod_{i=2}^n |\mc{b}_i|$;
$\mc{x}_n = \mc{a}_n (\mc{b}_n \times \mc{x}_{n-1})+\mc{c}_n$ for $n \geqslant 2$,
then
$$
 \mc{x}_n=(\mc{b}_{n,2} \times \mc{x}_1) \prod_{k=2}^n(\mc{b}_{n,k+1} \times \mc{a}_k) +
 \sum_{i=2}^n(\mc{b}_{n,i+1} \times \mc{c}_i)\prod_{k=i+1}^n(\mc{b}_{n,k+1} \times \mc{a}_k) \text{,}
$$
where $\mc{b}_{n,k} = \mc{b}_n \times \mc{b}_{n-1} \times \cdots \times \mc{b}_k$, if $k\leqslant n$
and
$\mc{b}_{n,k} = \mc{1}_1$, if $k > n$.
\end{lemma}

\pmn The proof is strightforward.

\pmn Vectors $\mc{r}_n'$ for similar nonhomogeneous sequence 
$x_n' = a_n x_{n-1}' + b_n x_{n-2}' + c_n$
such that $|(\mc{r}_n')_{2^{n-1}+1}^{2^n}|_1 = x_{n+1}'$,
are defined as follows: $\mc{r}_1' = \bv x_1 , x_2 \ev$, 
$\mc{r}_n' = (\mc{1}_2 \times \mc{r}_{n-1}')(\mc{h}_n \times \mc{1}_{2^{n-2}}) + c_n ( \mc{0}_{2^n-1} \cat \mc{1}_1  )$ 
for $n\geqslant 2$. To use the lemma we should, of course, 
do substitutions $\mc{a}_n=\mc{h}_n \times \mc{1}_{2^{n-2}}$ and 
$\mc{b}_{n,k} = \mc{1}_{2^{n-k+1}}$.

\begin{theorem}
If $w_0$ is arbitrary number, $a_{n,j}$ is arbitrary number sequence and $w_n = \sum_{j=0}^{n-1}a_{n,j}w_j$ for $n\in \mathbb{N}$, then
$$
\sum_{j=0}^nw_j = w_0 \sum_{\mc{v}\in \mathbb{V}_n}\prod_{k=1}^{|\mc{v}|}a_{(\mc{v})_k , (\mc{0}_1 \cat \mc{v})_k }\text{,}
$$
where set $\mathbb{V}_n$ consists of all vectors $\mc{v}$ such that $1 \leqslant(\mc{v})_{i-1} < (\mc{v})_i \leqslant n$ for $2 \leqslant i \leqslant |\mc{v}|$.
\end{theorem}

\begin{proof}Define vectors: $\mc{w}_0 = \bv w_0 \ev$, $\mc{w}_1 = \bv w_0 , a_{1,0} w_0 \ev$, 
$\mc{w}_n = \mc{w}_{n-1} \cat (\mc{w}_{n-1} \mc{q}_n)$ for $n\geqslant 2$, 
where $(\mc{q}_n)_1 = a_{n,0}$, $(\mc{q}_n)_{2^{k-1}+1}^{2^k} = a_{n,k} \mc{1}_{2^{k-1}}$
for $1 \leqslant k\leqslant n-1$.
From the recurrence equation it follows that if equality $|\mc{w}_l|_1 = \sum_{j=0}^l w_j$ is true for $l=n-1\geqslant 1$, then
it is true for $l=n$; it is true for $l=1$, so we conclude that it is true for any $l \in \mathbb{N}$.
\pmn Solving  the recurrence equation we get: $\mc{w}_n = w_0 \prod_{k=1}^n (\mc{1}_{2^{n-k}} \times (\mc{1}_{2^{k-1}} \cat \mc{q}_k)) $
for $n \in \mathbb{N} $. Noting that $\mc{q}_k = \bv a_{k,\lceil \log_2 j\rceil} \ev_{j=1}^{2^{k-1}}$ we have
$$
|\mc{w}_n|_1 = w_0 \sum_{j=1}^{2^n}\prod_{k=1}^n([2^{k-1},2^{k-1}]_j (a_{k,\lceil \log_2 ( 1+ (j-1)\bmod2^{k-1})\rceil} -1)  +1)  \text{.}
$$
\pmn The quantity $[2^k,2^k]_j$ equals the value of $k$th digit of number $(j-1)$ in binary numeral system.
If $k_i$ is serial number of $i$th digit $1$, then $\lceil \log_2 (1+{(j-1)\bmod 2^{k_i}})\rceil = 1+k_{i-1}$.
Therefore, $|\mc{w}_n|_1 = w_0 \sum_{j}\prod_{i}a_{k_i + 1,k_{i-1}+1}$,
assuming $k_{0}=-1$. Here $(k_i +1)$ ranges over $\mc{v} \in \mathbb{V}_n$: $k_i +1 = (\mc{v})_i$.\end{proof}

\pmn In the same way vectors $\mc{w}_n'$ for nonhomogeneous sequence 
$w_n' = c_n + \sum_{j=0}^{n-1}a_{n,j}w_j'$
such that $|\mc{w}_l'|_1 = \sum_{j=0}^l w_j$,
are defined as follows: 
$\mc{w}_0' = \bv w_0 \ev$, $\mc{w}_1' = \bv w_0 , c_1 + a_{1,0} w_0 \ev$, 
$\mc{w}_n' = \mc{w}_{n-1}' \cat (\mc{w}_{n-1}' \mc{q}_n) + c_n (\mc{0}_{2^n-1} \cat \mc{1}_1) $ for $n\geqslant2$.

\pbn
\section{Quadratic map}

It is well-known that in many cases iterations of a polynomial of degree 2 in the general case, i.e. solutions of quadratic map, can be expressed 
by iterations of a polynomial of degree 2 with one parameter.

\begin{theorem}
Let $p^{(0)}(x)=x$, $p^{(n)}(x)=p^{(n-1)}(p(x))$ $($ for $n \in \mathbb{N}$ $)$ be iterations of polynomial $p(x)=\lambda (x+1)x$, then
$$
p^{(n)}(x)=\sum_{k=1}^{2^n}x^k\sum_{i=(k-1)\omega_n}^{k \omega_n -1} \prod_{j=1}^n \lambda ^{\mu_{i,j-1}}\binom{\mu_{i,j-1}}{\mu_{i,j} - \mu_{i,j-1}}\text{,}
$$
where $\omega_n = 2^{\frac{n(n-1)}{2}}$, $\mu_{i,j} = \lceil \frac{1+ i \bmod \omega_{j+1}}{\omega_j} \rceil$.
\end{theorem}

\begin{proof}Any polynomial $p^{(n)}(x)$ can be expressed as follows: $p^{(n)}(x) = \sum_{k=1}^{2^n} g_{n,k}x^k$, where 
$g_{n,k}=g_{n,k}(\lambda)$ are polynomials defined by equalities:
$g_{0,1}=1$, $g_{n,k}=\sum_{i=1}^{2^{n-1}}q_{k,i} g_{n-1,i}$ (for $ 1 \leqslant k \leqslant2^n$), where
$q_{k,i}=\lambda^i \binom{i}{k-1}$.
\pmn Let $\mc{p}_0 = \bv 1 \ev$, $\mc{p}_n = (\mc{1}_{2^n} \times \mc{p}_{n-1}) (\mc{q}_n \times \mc{1}_{\omega_{n-1}})$
for $\in \mathbb{N}$, where
$$\mc{q}_n = (q_{1,j})_{j=1}^{2^{n-1}} \cat (q_{2,j})_{j=1}^{2^{n-1}} \cat \ldots \cat (q_{2^n,j})_{j=1}^{2^{n-1}}
= \bv q_{\lceil i/2^{n-1}\rceil, 1+ (i-1)\bmod 2^{n-1} }\ev_{i=1}^{2^{2n-1}} \text{.} $$
\pmn Then $|(\mc{p}_n)_{1+(k-1)\omega_n}^{k\omega_n}|_1 = g_{n,k}$. Solving  the equation we get: 
$$
\mc{p}_n = \prod_{k=1}^n \biggl( \mc{1}_{2^{\frac{(n+k+1)(n-k)}{2}}} \times \mc{q}_k \times \mc{1}_{\omega_{k-1}} \biggr) \text{.}
$$
\pmn Doing some calculations with $\mc{1}_{\infty} \times \mc{q}_k \times \mc{1}_{\omega_{k-1}}$, we obtain the result. \end{proof}

\pmn \emph{Remark}. Using generating polynomial
$$
|\mc{q}_k|_t=\lambda(1+t^m)\frac{(\lambda t^m + \lambda t^{2m+1})^m -1}{\lambda t^m + \lambda t^{2m+1}-1} \text{,}
$$
where $m=2^{k-1}$ and taking in account formula $|\mc{a} \times \mc{b}|_t =|\mc{a}|_{t^{|\mc{b}|}} |\mc{b}|_t$ we can
represent $p^{(n)}(x)$
as Hadamard product of $n$ functions. 

~\

Let's consider another episode.
Let $s^{(0)}(x)=x$, $s^{(n)}(x)=s^{(n-1)}(s(x))$ (for $n \in \mathbb{N} $)
be iterations of polynomial $s(x)= s_\lambda(x)=\lambda (x^2-1)+1$ and $\lambda \neq 0$.
Define vectors: $\mc{s}_1 = \bv x-1, 1\ev$, $\mc{s}_n = \lambda \mc{s}_{n-1}^{\langle 2 \rangle} - (\lambda -1) (\mc{0}_{2^{2^{n-1}}-1} \cat \mc{1})$
for $n \geqslant 2$, where triangular brackets indicate Kronecker degree. Obviously, $|\mc{s}_n|_1 = s(|\mc{s}_{n-1}|_1) = s^{(n)}(x)$.

\pmn Solving  the equation we get: 
$$
\mc{s}_n=\lambda \mc{s}_{n-1}^{\langle 2 \rangle} \mc{l}_{n-1} =
\lambda^{2^{n-1}-1} \mc{s}_1^{\langle 2^{n-1} \rangle} \mc{r}_{\lambda,n-1} =
\lambda^{2^{n-1}-1} \mc{s}_1^{\langle 2^{n-1} \rangle} (\mc{r_{\lambda}})_1^{2^{2^{n-1}}} \text{,}
\eqno{(1)}$$
where $\mc{l}_n = \mc{1}_{2^{2^n}-1} \cat \bv \lambda^{-1} \ev$,  
$\mc{r}_{\lambda,n} = \prod_{i=1}^{n}\mc{l}_i^{\langle 2^{n-i}\rangle}$,  
$\mc{r_\lambda} = \prod_{i=1}^{\infty}\mc{l}_i^{\langle \infty \rangle}$.  
\pmn Therefore, we have
$$
s_\lambda^{(n)} (x) = \lambda^{2^{n-1}-1}\sum_{j=1}^{2^{2^n}}(x-1)^{(\mc{h}_n)_j} \lambda^{-(\log_2 \mc{r}_{2,n-1})_j}\text{,}
\eqno{(2)}$$
where $\mc{h}_n = \log_2\bv 2,1 \ev^{\langle 2^{n-1} \rangle}$.
And it can be easily shown that 
$$(\mc{h}_n)_j = 2^{n-1}-\sum_{k=0}^{\infty}[2^k,2^k]_j \text{   and   } 
(\log_2 \mc{r}_{2,n})_j  = \sum_{k,i=0}^{\infty}[2^{2^i k}(2^{2^i}-1),2^{2^i k}]_j 
$$
by using simple formula 
$$(\log_2 (\mc{1}_{m-1} \cat \bv 2  \ev)^{\langle n \rangle})_j = \sum_{k=0}^{n-1}[m^k(m-1),m^k]_j \text{    (for } j\leqslant m^n \text{).} $$
\pmn Substituting $2$ for $x$ in (2), we have
$$
s_\lambda^{(n)} (2) = \lambda^{2^{n-1}-1}\sum_{k=1}^{2^{n-1}}\kappa_{n,k}{\lambda}^{-k}\text{,}
$$
where $\kappa_{n,k}$ denotes the number of elements in $\log_2 \mc{r}_{2,n-1}$ that equal to $k$.
This function can be defined recursively as follows: 
$$\kappa_{n,0} = 3^{2^{n-1}} ,
\kappa_{1,k} = \delta_{k,1} ,
\kappa_{n,k} =\delta_{k,2^{n-1}} - \delta_{k,2^{n-2}} + \sum_{i=0}^k \kappa_{n-1,k-i} \kappa_{n-1,i} \text{.}
$$
\pmn Evaluating $\kappa_{n,k}$ we have:
$$\kappa_{n,1} = 2^{n-1} 3^{2^{n-1}-1},
\kappa_{n,2} = -\kappa_{n,1} (\frac{1}{12}-\sum_{i=0}^{n-1}\kappa_{i+1,1}({\kappa_{i,1}}^2 + \delta_{k,2^{n-1}} - \delta_{k,2^{n-2}})
$$
\pmn and so on.

\pmn \emph{Remark}. Replacing $2^{n-1}$ by $n$ in the last expression of (1) we get new sequence of vectors:
$\mc{s}_n' = \lambda^{n-1} \mc{s}_1^{\langle n \rangle} (\mc{r})_1^{2^n}$.
Let $f_n(x) = |\mc{s}_n'|_1$. Hypothesis:
$$
f_n(x)=
\begin{cases}
f_{n/2}(s(x)) \text{, if } n \text{ is even} \\
\lambda x f_{n-1}(x) \text{, if } n \text{ is odd}\\
x \text{, if } n=1
\end{cases}
$$

\vspace{5mm}

\vspace{5mm}
\begin{thebibliography}{9}

\bibitem{Pop1987}Popenda J., \emph{One expression for the solutions of second order difference equations}, Proc. Amer. Math. Soc., 1987, 100(1), 87-93

\bibitem{Kit1993}Kittappa R., \emph{A representation of the solution of the nth order linear difference equation with variable coefficients}, Linear Algebra Appl., 1993, 193, 211-222  

\bibitem{Mal2000}Mallik R., \emph{On the solution of a linear homogeneous difference equation with variable coefficients},  SIAM J.Math. Anal., 2000, 31 (2), 375-385

\bibitem{All2003}Allouche J., Shallit J., \emph{Automatic Sequences: Theory, Applications, Generalizations}, Cambridge University Press, 2003, 588 с.

\bibitem{Lan2007}Lando S., \emph{Lectures on Generating Functions}, American Mathematical Society, 2003

\bibitem{Voe1984}Voevodin V., Kuznetsov Y., \emph{Matrices and calculations}, Nauka, 1984 (in Russian)

\bibitem{Fomin1980}Fomin S., \emph{Number systems}, The University of Chicago Press, 1975

\end{thebibliography}
\end{document}